\documentclass[10pt]{amsart}
\usepackage{geometry} % see geometry.pdf on how to lay out the page. There's lots.
\usepackage{breqn}  

\usepackage{amsmath, amsthm, amssymb}

\numberwithin{equation}{section}

\newtheorem{thm}{Theorem}[section]

\newtheorem{lem}[thm]{Lemma}
\newtheorem{prop}[thm]{Proposition}

\newcommand{\Ric}{\text{Ric}}
\newcommand{\Hess}{\text{Hess}\hspace{0.5 mm}}
\newcommand{\dvol}{\text{dvol}}
\newcommand{\R}{\mathbb{R}}
\newcommand{\diam}{\textrm{diam}\hspace{0.5 mm}}
\newcommand{\csch}{\textrm{csch}}
\newcommand{\vol}{\text Vol}

\newcommand{\grad}{\nabla}

\newcommand{\ob}{\mathbf{b}_\pm}
\newcommand{\Mx}{\textrm{Mx}}
\newcommand{\Capa}{\textrm{Cap}}
\newcommand{\Cov}{\textrm{Cov}}

\def\Xint#1{\mathchoice
{\XXint\displaystyle\textstyle{#1}}%
{\XXint\textstyle\scriptstyle{#1}}%
{\XXint\scriptstyle\scriptscriptstyle{#1}}%
{\XXint\scriptscriptstyle\scriptscriptstyle{#1}}%
\!\int}
\def\XXint#1#2#3{{\setbox0=\hbox{$#1{#2#3}{\int}$ }
\vcenter{\hbox{$#2#3$ }}\kern-.6\wd0}}

\def\dashint{\Xint-}

\title{Fundamental Groups of Spaces with Bakry-Emery Ricci Tensor Bounded Below}
\author{Maree Jaramillo}
\date{} % delete this line to display the current date

\begin{document}

\begin{abstract}  We first extend the Cheeger-Colding Almost Splitting Theorem to smooth metric measure spaces.  Arguments utilizing this extension of the Almost Splitting Theorem show that if a smooth metric measure space has almost nonnegative Bakry-Emery Ricci curvature and a lower bound on volume, then its fundamental group is almost abelian.  Second, if the smooth metric measure space has Bakry-Emery Ricci curvature bounded from below then the number of generators of the fundamental group is uniformly bounded.  These results are extensions of theorems which hold for Riemannian manifolds with Ricci curvature bounded from below.  
\end{abstract}

\maketitle

\section{Introduction}

A smooth metric measure space is a triple $(M^n, g, e^{-f} \dvol_g)$, where $M^n$ is a complete $n$-dimensional Riemannian manifold equipped with metric $g$ and volume density $\dvol_g$.  The potential function $f: M^n \to \R$ is smooth. Smooth metric measure spaces occur naturally as collapsed measured Gromov-Hausdorff limits of sequences of warped products $(M^n \times F^m, g_\epsilon, \widetilde{\dvol_{g_\epsilon}})$ as $\epsilon \to 0$, where $\widetilde{\dvol_{g_\epsilon}}$ is the renormalized Riemannian measure and $g_\epsilon = g_M + (\epsilon e^{\frac{-f}{m}})^2 g_F$ is the warped product metric with $g_M$ and $g_F$ the metrics on $M$ and $F$, respectively.   The Ricci curvature of the warped product metric $g_\epsilon$ in the $M$ direction is given by
$$\Ric^m_f = \Ric + \Hess f - \frac{1}{m}df \otimes df.$$
This leads to the definition of the  $m$-Bakry Emery Ricci tensor on the limit space as
$$ \Ric^m_f = \Ric + \Hess f - \frac{1}{m} df \otimes df, \hspace{6 mm} 0 < m \leq \infty.$$
When $m= \infty$, we have the Bakry-Emery Ricci tensor on $(M^n, g, e^{-f}\dvol_g)$ given by
$$\Ric_f = \Ric + \Hess f.$$
Thus, the Bakry-Emery Ricci tensor is a natural analogue to Ricci curvature on $(M^n, g, e^{-f}\dvol_g)$.  
This tensor has appeared in the work on diffusion processes by Dominique Bakry and Michel Emery.  
Moreover, it occurs in the study of Ricci flow, which was utilized most notably by Grigori Perelman in his proof of the Poncar\'e conjecture.

Since topological and geometric information can be obtained for manifolds with Ricci curvature bounded from below and $\Ric_f = \Ric$ when $f$ is constant, it is natural to ask if the same information holds true for smooth metric measure spaces with Bakry-Emery Ricci tensor bounded from below.  
  Indeed, this very question has been at the center of an active field of research studied by many. In particular, Guofang Wei and Will Wylie have shown that when $\Ric_f$ is bounded from below and, in addition, either $f$ is bounded or $\partial_r f \geq -a$ for $a \geq 0$ along minimal geodesics from a fixed $p \in M$, then the mean curvature and volume comparison theorems can be extended to the smooth metric measure space setting \cite{WW}.

One important result which has already been extended to the smooth metric measure space setting is the  Cheeger-Gromoll Splitting Theorem (see \cite{FLZ}, \cite{Li1}, 
\cite{WW}).  In this paper, we discuss the extension of the quantitative version of this theorem, the Cheeger-Colding Almost Splitting Theorem, to the smooth metric measure space setting.  

A crucial step in the proof of the Cheeger-Gromoll Splitting Theorem is to construct a function $b$ such that $|\grad b| = 1$ and $\Hess b \equiv 0$.  In the proof of the Cheeger-Colding Almost Splitting Theorem, one constructs a harmonic function $\mathbf{b}$ whose Hessian is small in the $L^2$-sense.  In order to extend the Almost Splitting Theorem to smooth metric measure spaces following Cheeger-Colding's proof, we also construct $f$-harmonic functions $\mathbf{b}_\pm$ and obtain for these functions an $L^2$-Hessian estimate with respect to the conformally changed volume density $e^{-f}\dvol_g$.  
\begin{thm}\label{HessEst}
Given $R > 0$, $L > 2R + 1$ and $\epsilon >0$, let $p, q_+, q_- \in M^n$.  If $(M^n, g, e^{-f}\dvol_g)$ satisfies
\begin{equation} \label{hyp1}
|f| \leq k,
\end{equation}
\begin{equation} Ric_f \geq -(n-1)H \hspace{5 mm} (H \geq 0),
\end{equation}
\begin{equation} 
 \min\{d(p, q_+), d(p, q_-)\} \geq L, 
\end{equation}
\begin{equation}  \label{hyp2}
e(p)  \leq \epsilon, 
\end{equation}
then
\begin{equation}\label{Hess}
\dashint_{B(p, \frac{R}{2})} | \Hess \ob|^2 e^{-f} \dvol_g \leq \Psi(H, L^{-1}, \epsilon | k, n, R).
\end{equation}
\end{thm}
The functions $e$, $\Psi$, and $\mathbf{b}_\pm$ in Theorem \ref{HessEst} are defined in \eqref{excess}, \eqref{Psi}, and \eqref{b}, respectively.  With the $L^2$-Hessian estimate \eqref{Hess}, one may then obtain a type of Pythagorean Theorem from which the Almost Splitting Theorem will follow.  

We note here that while preparing this paper for publication, the author was notified that Feng Wang and Xiaohua Zhu posted \cite{WZ} in which they also develop an  $L^2$-Hessian estimate and Almost Splitting Theorem for Bakry-Emery Ricci tensor bounded from below.   In Wang and Zhu's version, the Hessian estimate assumes that the gradient of the potential function, rather than the potential function itself, is bounded. Wang and Zhu's Almost Splitting Theorem \cite[Theorem 3.1]{WZ} assumes that both the potential functions and their gradients are bounded, whereas Theorem \ref{Splitting} assumes only that the potential functions are bounded.

The hypotheses for the Hessian estimates and Almost Splitting Theorem of these two papers differ due in part to the gradient estimates used by the authors.  
In Cheeger-Colding's proof of the $L^2$-Hessian estimate in the Riemannian setting,  a cutoff function $\phi$ is used.  This cutoff function $\phi$ has the property that $|\grad \phi|$ and $|\Delta \phi|$ are bounded by constants depending only on $n, H, R$.  
  The construction of this cutoff function in the Riemannian case relies on the gradient of a function being bounded away from the boundary of a ball.  This boundedness is guaranteed by the Cheng-Yau gradient estimate.    The gradient estimate obtained by Wang and Zhu requires that the gradient of the potential function is bounded.  The gradient estimate obtained here requires that the potential function itself is bounded.

  \vspace{3 mm}

\begin{thm}\label{GradEst2}
Let $(M^n, g, e^{-f} dvol)$ be a complete smooth metric measure space with $|f| \leq k$ and $\Ric_f \geq -(n-1)H^2$ where $H \geq 0$.  If $u$ is a positive function defined on $\overline{B(q, 2R)}$ with $\Delta_f u= c$, for a constant $c \geq 0$, then for any $q_0 \in \overline{B(q, R)}$, we have
$$|\grad u| \leq \sqrt{c_1(n, k, H, R) \sup_{p \in B(q; 2R)} u(p)^2 + c_2(c, n) \sup_{p \in B(q; 2R)} u(p)}.$$
\end{thm}

Theorem \ref{GradEst2} follows directly from Theorem \ref{GradEst}, which is an extension of Kevin Brighton's gradient estimate for $f$-harmonic functions in \cite{Bri}, to functions $u$ such that $\Delta_f u = c$ for any nonnegative constant $c$.  We are not sure if the estimate holds true when $c < 0$.

%%%%%%%%%%%%%%%%%%%%%%%

The Almost Splitting theorem will allow us to generalize the arguments in the proofs of two results on the fundamental group of Riemannian manifolds to the smooth metric measure space setting.  The first result is an extension of a theorem of Gabjin Yun \cite{Yun}.  Yun's result asserts that the fundamental group of a Riemannian manifold with almost nonnegative Ricci curvature, diameter bounded from above, and volume bounded from below is almost abelian.   This result is a strengthening of a theorem of Wei \cite{Wei} which shows that under the same conditions $\pi_1(M)$ has polynomial growth.  In order to extend Yun's theorem, we first develop an absolute volume comparison, Proposition \ref{AbsVolComp}, which allows us to extend Wei's theorem to smooth metric measure spaces.  We then follow Yun's argument, utilizing the Almost Splitting Theorem,  to obtain the following result.  

   \begin{thm}\label{AlmostAbelian}
   For any constants $D, k, v > 0$, there exists $\epsilon = \epsilon(D, k, n, v) > 0$ such that if a smooth metric measure space $(M^n, g, e^{-f}dvol_g)$ with $|f| \leq k$ admits a metric under which it satisfies the conditions 
   \begin{equation}\label{Ric}
   \Ric_f \geq -\epsilon,
   \end{equation} 
   \begin{equation}\label{diam}
   \diam(M) \leq D,
   \end{equation} 
   \begin{equation}\label{vol}
   \vol_f(M) \geq v,
   \end{equation}
    then $\pi_1(M)$ is almost abelian, i.e. $\pi_1(M)$ contains an abelian subgroup of finite index.  
   \end{thm}

   The second result is an extension of Vitali Kapovitch and Burkhard Wilking's theorem which gives a uniform bound on the number of generators of $\pi_1(M)$ for the class of $n$-dimensional manifolds $M^n$ with $\Ric \geq -(n-1)$ and $\diam (M, g) \leq D$, for given $n$ and $D$ \cite[Theorem 3]{KW}.   A uniform bound had been given previously in the case when the conjugate radius is bounded from below \cite{WeiBettiNum}.  An extension of this theorem to the smooth metric measure space setting is as follows.

   %asserts that if $\pi_1(M, p)$ is generated by loops of length $\leq R$ and $\Ric \geq -(n-1)$ on $B(p, 2R)$ and furthermore $\overline{B(p, 2R)}$ is compact, then the number of generators of $\pi_1(M, p)$ is uniformly bounded by a constant depending on $n$ and $R$  \cite[Theorem 2.5]{KW}.   This was proven in the case when the conjugate radius is bounded from below in \cite{WeiBettiNum}.  An extension of this theorem to the smooth metric measure space setting is as follows.

   \begin{thm}\label{UnifBound}
   Given $n$, $D$, and $k$, there is a constant $C = C(n,D,k) >0$ such that the following holds.  Let $(M^n, g, e^{-f}\dvol_g)$ be a smooth metric measure space, $|f| \leq k$, $\diam M \leq D$ and $\Ric_f \geq -(n-1)$.  Then $\pi_1(M)$ is generated by at most $C$ elements.
   \end{thm}

\vspace{2 mm}
\noindent {\bf Remark} When $m \neq \infty$ and $\Ric^m_f$ is bounded from below, comparison theorems hold with no additional assumptions on $f$ \cite{Qi}, see also \cite{WW2}.  Due to this fact, there are versions of our theorems for $\Ric^m_f$ bounded from below with no additional assumptions on $f$.

\vspace{2 mm}
The remaining sections of this paper are structured as follows.  In Section 2, we discuss the extension of the Almost Splitting Theorem, along with the essential tools which allow one to extend to the smooth metric  measure space setting.    These essential tools include the $f$-Laplacian Comparison, Gradient Estimate, Segment Inequality, and the Quantitative Maximum Principle for smooth metric measure spaces.   We will also provide a proof of the key Hessian estimate, Theorem \ref{HessEst}.  In Section 3, we develop an absolute volume comparison which is used to extend Wei's theorem to the smooth metric measure space setting.   In Sections 4 and 5 we discuss the proofs of Theorems \ref{AlmostAbelian} and \ref{UnifBound}, respectively.

\section*{Acknowledgement}
The author would like to thank her adviser Guofang Wei for her guidance and many helpful discussions and suggestions which led to the completion of this paper.

   \section{The Almost Splitting Theorem for Smooth Metric Measure Spaces}

As indicated in the introduction, an essential tool in establishing the Almost Splitting Theorem for smooth metric measure spaces will be the Hessian estimate, Theorem \ref{HessEst}.  In order to obtain the original estimate, other essential tools, including the Laplace Comparison, Cheng-Yau Gradient Estimate, and Abresch-Gromoll Inequality, were utilized. Cheeger and Colding also developed key tools, such as the Segment Inequality.  With the extension of such tools to the smooth metric measure space setting, we can generalize the arguments of Cheeger and Colding \cite{CC}, see also \cite{Ch}.

In order to extend the essential tools mentioned above, one must integrate with respect to the measure $e^{-f}\dvol_g$.  In addition, one must replace the Laplace-Beltrami operator on Riemannian manifolds with its natural analog on smooth metric measure spaces.  This analog is the $f$-Laplacian, defined for functions $u \in C^2(M)$ by 
$$\Delta_f(u) = \Delta(u) - \langle \nabla u, \nabla f \rangle.$$  This operator is natural in the sense that it is self-adjoint with respect to the measure $e^{-f} \dvol_g$.

 From the Mean Curvature Comparison \cite[Theorem 1.1]{WW} and definition of the $f$-Laplacian, one immediately obtains the following $f$-Laplacian comparison.

\begin{prop}\label{fLapComp}{\bf ($f$-Laplacian Comparison)}
Suppose $Ric_f \leq (n-1)H$ with $|f| \leq k$.  Let $\Delta_H^{n+4k}$ denote the Laplacian of the simply connected model space of dimension $n+ 4k$ with constant sectional curvature $H$.  Then for radial functions $u$, 
\begin{enumerate}
\item  $\Delta_f(u) \leq \Delta_H^{n+4k} u$ if $u' \geq 0$.
\item $\Delta_f(u) \geq \Delta_H^{n+4k} u$ if $u' \leq 0$.
\end{enumerate}
\end{prop}
Using the definition of the $f$-Laplacian and the classical Bochner formula on Riemannian manifolds, one also immediately obtains a Bochner formula for smooth metric measure spaces
\begin{equation}\label{Bochner}
\frac{1}{2}\Delta_f(|\grad u|^2) = |\Hess u|^2 + \langle \nabla u, \nabla \Delta_f u \rangle + \Ric_f(\nabla u, \nabla u),
\end{equation}
 for any $u \in C^3(M)$.  This Bochner formula allows us to obtain the following gradient estimate.

%%%%%%%%%%%%%%%%%%%%%%%%%%%%%%%%%%%%%

\begin{prop}\label{GradEst}{\bf (Gradient Estimate)}
Let $(M^n, g, e^{-f} dvol)$ be a complete smooth metric measure space with $\Ric_f \geq -(n-1)H^2$ where $H \geq 0$.  If $u$ is a positive function defined on $\overline{B(q, 2R)}$ with $\Delta_f u= c$, $c \geq 0$, then for any $q_0 \in \overline{B(q, R)}$, we have
$$|\grad u| \leq \sqrt{c_1(\alpha, n, H, R) \sup_{p \in B(q; 2R)} u(p)^2 + c_2(c, n) \sup_{p \in B(q; 2R)} u(p)}$$
where 
$\alpha = \max_{p \in p: d(p,q) = r_0} \Delta_f r(p)$ for any $r_0 \leq R$ and $r(p) = d(p, q)$.  
\end{prop}

Note that in the above gradient estimate we make no assumption on the potential function $f$. We include a sketch of the proof below, which modifies the proof of Brighton's gradient estimate for $f$-harmonic functions in \cite{Bri} to consider the case of $\Delta_f u = c$, for a positive constant $c$.  

\begin{proof}
Let $h = u^\epsilon$ where $\epsilon \in (0,1)$.  Applying \eqref{Bochner} to $h$ gives
$$ \frac{1}{2} \Delta_f |\nabla h|^2 = |\Hess h|^2 + \langle \nabla h, \nabla(\Delta_f h) \rangle + \Ric_f(\nabla h, \nabla h).$$

Using the Schwartz inequality, we have
\begin{align*}
|\Hess h|^2 & \geq \frac{|\Delta h|^2}{n}
\\ & = \frac{1}{n}(\Delta_f h + \langle \nabla f, \nabla h \rangle)^2
\\ & = \frac{1}{n}\left(\epsilon u^{\epsilon-1} \Delta_f u + \frac{(\epsilon-1)|\grad h|^2}{\epsilon h} + \langle \grad f, \grad h \rangle \right)^2
\\ & = \frac{1}{n}\left(\epsilon u^{\epsilon-1} c + \frac{(\epsilon-1)|\grad h|^2}{\epsilon h} + \langle \grad f, \grad h \rangle \right)^2
\end{align*}
where in the last equality we used the fact that $\Delta_f u = c > 0$.  
This, together with the lower bound on the Bakry-Emery Ricci tensor give

\begin{equation}\label{Prelim}
\begin{split}
\frac{1}{2} \Delta_f |\nabla h|^2  \geq
& \frac{(\epsilon-1)^2}{\epsilon^2 h^2 n}|\grad h|^4+ \frac{2c(\epsilon-1)}{h^{1/\epsilon}n}|\grad h|^2 + \frac{2(\epsilon-1)}{\epsilon h n}|\grad h|^2\langle \grad f, \grad h\rangle + \frac{\epsilon^2c^2}{n}(h^{2-2/\epsilon}) 
\\ &  + \frac{2c\epsilon}{n}(h^{1-1/\epsilon}) \langle \grad f, \grad h \rangle + \frac{1}{n}\langle \grad f, \grad h \rangle^2 + \frac{(\epsilon - 1)}{\epsilon h}\langle \grad h, \grad|\grad h|^2\rangle 
\\ &  -\frac{(\epsilon-1)}{\epsilon h^2}|\grad h|^4 + \frac{c(\epsilon-1)}{h^{1/\epsilon}}|\grad h|^2 - (n-1)H^2|\grad h|^2
\end{split}
\end{equation}

In order to control the mixed term $2\frac{(\epsilon-1)}{\epsilon h n}|\grad h|^2\langle \grad f, \grad h \rangle$ in \eqref{Prelim}, we consider two cases according to whether $|\grad h|^2$ dominates over $\langle \grad h, \grad f \rangle$, or vice versa.  In the first case, suppose that  $p \in \overline{B(q, 2R)}$ such that $\langle \grad h, \grad f \rangle \leq a\frac{|\grad h|^2}{h}$ for some $a > 0$ to be determined. At this point we have 

\begin{align}\label{Case1}
\begin{split}
\frac{1}{2}\Delta_f |\grad h|^2 a& \geq \left[\frac{(\epsilon-1)^2 + 2\epsilon(\epsilon-1)a - \epsilon(\epsilon-1)n}{\epsilon^2 n}	\right]\frac{|\grad h|^4}{h^2} + \left[\frac{c(\epsilon-1)(2+n)}{n}	\right]\frac{|\grad h|^2}{h^{1/\epsilon}}
\\ & \hspace{5 mm} + \frac{1}{n}(\epsilon c h^{1-1/\epsilon} + \langle \grad f, \grad h\rangle)^2 + \frac{\epsilon-1}{\epsilon h}\langle \grad h, \grad |\grad h|^2 \rangle -(n-1)H^2 |\grad h|^2 
\\ & \geq \left[\frac{(\epsilon-1)^2 + 2\epsilon(\epsilon-1)a - \epsilon(\epsilon-1)n}{\epsilon^2 n}	\right]\frac{|\grad h|^4}{h^2} + \left[\frac{c(\epsilon-1)(2+n)}{n}	\right]\frac{|\grad h|^2}{h^{1/\epsilon}}
\\ & \hspace{5 mm} + \frac{\epsilon-1}{\epsilon h}\langle \grad h, \grad |\grad h|^2 \rangle -(n-1)H^2 |\grad h|^2
\end{split}
\end{align}

In the case that $p \in \overline{B(q, 2R)}$ such that $\langle \grad h, \grad f \rangle \geq a\frac{|\grad h|^2}{h}$, we have
\begin{align}\label{Case2}
\begin{split}
\frac{1}{2}\Delta_f|\grad h|^2 & \geq  \left[\frac{(\epsilon-1)^2 - \epsilon(\epsilon - 1)n}{\epsilon^2 n}  \right]\frac{|\grad h|^4}{h^2} + \left[\frac{c(\epsilon-1)(2+n)}{n}	\right]\frac{|\grad h|^2}{h^{1/\epsilon}} 
\\ & \hspace{5 mm}  +\left[\frac{2(\epsilon-1) + \epsilon a}{\epsilon n a}	\right]\langle \grad f, \grad h \rangle^2+\frac{\epsilon^2 c^2}{n}(h^{2 - 2/\epsilon}) + \frac{2c\epsilon a}{nh^{1/\epsilon}}|\grad h|^2 
\\ & \hspace{5 mm} + \frac{\epsilon-1}{\epsilon h}\langle \grad h, \grad |\grad h|^2 \rangle -(n-1)H^2 |\grad h|^2 
\\ & \geq \left[\frac{(\epsilon-1)^2 - \epsilon(\epsilon - 1)n}{\epsilon^2 n}  \right]\frac{|\grad h|^4}{h^2} + \left[\frac{c(\epsilon-1)(2+n)}{n}	\right]\frac{|\grad h|^2}{h^{1/\epsilon}} 
\\ & \hspace{5 mm} +\left[\frac{2(\epsilon-1) + \epsilon a}{\epsilon n a}	\right]\langle \grad f, \grad h \rangle^2
 + \frac{\epsilon-1}{\epsilon h}\langle \grad h, \grad |\grad h|^2 \rangle -(n-1)H^2 |\grad h|^2 
\end{split}
\end{align}
Note that in \eqref{Case2} the assumption that $c \geq 0$ is necessary to have $\frac{2c\epsilon a}{nh^{1/\epsilon}}|\nabla h|^2 \geq 0$ which allows us to obtain the second inequality.

As in Brighton's proof, we see that choosing $\epsilon= \frac{7}{8}$ and $a = \frac{1}{2}$ will make the coefficient of the $\frac{|\grad h|^4}{h^2}$ term positive in both cases.  This choice also gives a positive coefficient of the $\langle \grad f, \grad h \rangle^2$ term in the second case.  With this choice of $\epsilon$ and $a$, we see that for every $p \in \overline{B(q, 2R)}$, we have
\begin{equation}\label{PrelimEst}
\frac{1}{2}\Delta_f |\grad h|^2 \geq \frac{7n-6}{49n}\frac{|\grad|h|^4}{h^2} - \frac{c(2+n)}{8n}\frac{|\grad h|^2}{h^{8/7}} -\frac{1}{7h}\langle \grad h, \grad|\grad h|^2 \rangle - (n-1)H^2|\grad h|^2.
\end{equation}
Let $g: [0,2R] \to [0,1]$ have the properties
\begin{itemize}
\item $g|_{[0,2R]} = 1$
\item supp$(g) \subseteq [0,2R)$
\item $\frac{-K}{R}\sqrt g \leq g' \leq 0$
\item $|g''| \leq \frac{K}{R^2}$
\end{itemize}
where the last two properties hold for some $K > 0$.  Define $\phi: \overline{B(q, 2R)} \to [0,1]$ by $\phi(x) = g(d(x, q))$.  Set $G = \phi|\grad h|^2$.  Then \eqref{PrelimEst} can be written as
\begin{equation}\label{SecondEst}
\begin{split}
\frac{1}{\phi} \Delta_f G \geq \frac{G}{\phi^2} \Delta_f \phi &+ 2 \langle \frac{\grad \phi}{\phi}, \frac{\grad G}{\phi} - \frac{\grad \phi}{\phi^2} G \rangle + \frac{14n-12}{49nh^2}\frac{G^2}{\phi^2} 
\\ &- \frac{c(2+n)}{4nh^{8/7}}\frac{G}{\phi} 
- \frac{2}{7h}\langle \grad h, \frac{\grad G}{\phi} - \frac{\grad \phi}{\phi^2} G \rangle-2(n-1)H^2\frac{G}{\phi}.
\end{split}
\end{equation}

Next, we consider the point $q_0 \in \overline{B(q, 2R)}$ at which $G$ achieves its maximum.  At such a point, \eqref{SecondEst} can be rewritten as
\begin{equation}\label{MaxPointEst}
\frac{14n-12}{49nh^2}G \leq -\Delta_f \phi + 2\langle \frac{\grad \phi}{\phi}, \grad \phi \rangle + \frac{c(2+n)}{4nh^{8/7}} \phi - \frac{2}{7h}\langle \grad h, \grad \phi \rangle + 2(n-1)H^2 \phi.
\end{equation}

If $q_0 \in B(q, R)$, then  \eqref{MaxPointEst} can be rewritten as
\begin{equation}\label{EasyGradEst}
|\grad u|^2 \leq \frac{8c(2+n)}{7n-6}u + \frac{64n(n-1)}{7n-6}H^2 u^2
\end{equation}
when evaluated at $q_0$.

If $q_0 \in \overline{B(q,2R)} \setminus B(q, R)$, one uses the mean curvature comparison \cite[Theorem 2.1]{WW}, the properties of $\phi$, and \eqref{MaxPointEst} to obtain
\begin{equation}\label{HardGradEst}
|\grad u|^2 \leq \frac{64}{13n-12}\left[\frac{K\alpha R + K + 3K^2}{R^2} + 2(K+1)(n-1)H^2 \right]u^2 + \frac{16c(2+n)}{n(13n-12)} u
\end{equation}
at the point $q_0$.

Taking the supremum of $u$ and $u^2$ over $\overline{B(q, 2R)}$ in \eqref{EasyGradEst} and \eqref{HardGradEst} yields the desired form of the gradient estimate.  
\end{proof}

Since for our purposes the potential function is bounded, that is $|f|\leq k$, we may use the $f$-Laplacian comparison, Proposition \ref{fLapComp}, to modify the above gradient estimate, Proposition \ref{GradEst}, slightly.  By setting $r_0 = R/2$ we may apply the $f$-Laplacian comparison directly to $\alpha$ to obtain
$$\alpha \leq (n+4k-1)H \coth(HR/2).$$
Alternatively, we may use Proposition \ref{fLapComp} in place of the mean curvature comparison \cite[Theorem 2.1]{WW} when obtaining \eqref{HardGradEst}.  In that case,
$$ \Delta_f r(q_0) \leq (n+4k-1)H \coth (Hr(q_0)) \leq (n+4k-1)H \coth (HR)$$
since $R \leq r(q_0) \leq 2R$.  
In either case, one will obtain a gradient estimate as in Theorem \ref{GradEst2} which is no longer dependent upon $\alpha$.  

Note that this gradient estimate only holds for nonnegative $c$.  If we consider the case $c < 0$, the term $\frac{2c\epsilon}{n}(h^{1-1/\epsilon}) \langle \grad f, \grad h \rangle$ in \eqref{Prelim} becomes problematic.  In particular, when $p \in \overline{B(q, 2R)}$ is such that $\langle \grad h, \grad f \rangle$ dominates over $|\grad h|^2$, we replace the term $\frac{2(\epsilon-1)}{\epsilon h n}|\grad h|^2\langle \grad f, \grad h \rangle$ by $\frac{2(\epsilon - 1)}{\epsilon h n a} \langle \grad f, \grad h \rangle$.  In order to control this term, we group it with $\frac{1}{n}\langle \grad f, \grad h \rangle^2$.  Then we can no longer group the $\langle \grad f, \grad h\rangle$ term with other terms to create a perfect square, as in \eqref{Case1}.  Moreover, since its coefficient is negative, we must keep this term in the estimate.  Thus, without any additional assumptions, such as a bound on $|\grad f|$, there is no way to control this term.  As noted in the introduction, the assumption of a bound on $|f|$ rather than a bound on $|\grad f|$ in this gradient estimate is one of the reasons for the difference in hypotheses of the Almost Splitting Theorem of Wang and Zhu \cite[Theorem 3.1]{WZ} and Theorem \ref{Splitting}.

Finally, in order to convert estimates of integrals of functions over a ball to estimates  of integrals of functions along a geodesic segment,  we need the following Segment Inequality.

\begin{prop}\label{SegmentInequality}{\bf (Segment Inequality)}
Let $(M^n, g, e^{-f}\dvol_g)$ be a complete metric measure space with $Ric_f \geq (n-1)H$ and  $|f(x)| \leq k$.  
Let $A_1, A_2 \in M^n$ be open sets and assume for all $y_1 \in A_1$, $y_2 \in A_2$, there is a minimal geodesic, $\gamma_{y_1, y_2}$ from $y_1$ to $y_2$, such that for some open set, $W$, 
$$ \bigcup_{y_1, y_2} \gamma_{y_1, y_2} \subset W.$$
If $v_i$ is a tangent vector at $y_i$, $i = 1,2$, and $|v_i| = 1$, set
$$I(y_i, v_i)=  \{t | \gamma(t) \in A_{i+1} , \gamma |[0,t] \text{ is minimal} , y'(0) = v_i\}.$$
Let $|I(y_i, v_i)|$ denote the measure of $I(y_i, v_i)$ and put
$$\mathcal{D}(A_i, A_{i+1} )= \sup_{y_1, y_2} |I(y_i, v_i)|.$$
Here $A_{2+1} := A_1$.  
Let $h$ be a nonnegative integrable function on $M$.  Let $D = \max d(y_1,y_2)$.  Then 

\begin{align*}
 \int_{A_1 \times A_2} &\int_0 ^{\overline{y_1, y_2}} h(\gamma_{y_1,y_2}(s))ds (e^{-f}\dvol_g)^2 \leq
 \\ & c(n + 4k, H, D)[\mathcal{D}(A_1, A_2) Vol_f(A_1) + \mathcal{D}(A_2, A_1) Vol_f(A_2)] \times \int_W h e^{-f}\dvol_g
 \end{align*}

where $c(n+4k, H, D) = \sup_{0 < s/2 \leq u \leq s} \mathcal{A}_H^{n+4k}(s)/\mathcal{A}_H^{n+4k}(u)$, where $\mathcal{A}_H^{n+4k}(r)$ denotes the area element on $\partial B(r)$ in the model space with constant curvature $H$ and dimension $n+4k$.  
\end{prop}

To obtain this result for smooth metric measure spaces one may follow the arguments of the proof in the original setting as given by Cheeger and Colding in \cite[Theorem 2.11]{CC}.   %Cheeger in \cite[Theorem 2.15]{Ch}.  
In the smooth metric measure space setting, integrals are computed with respect to the conformally changed volume element, $e^{-f}\dvol_g$, and we use the volume element comparison which follows from \cite[Theorem 1.1]{WW}.

Finally, the Abresch-Gromoll Quantiatative Maximal Principle  was necessary in the proof of the Abresch-Gromoll inequality and also in obtaining an appropriate cutoff function needed to prove the Hessian estimate.   Since this proof varies slightly from the exposition contained in \cite{AG} or \cite{Ch}, we retain the proof here.

\begin{prop}\label{QuantMaxP}{\bf (Quantitative Maximal Principle)}
If $Ric_f \geq (n-1)H$, $(H \leq 0)$, $|f| \leq k$ and $U: \overline{B(y,R)} \subset M^n \to \R$ is a Lipschitz function with

\begin{enumerate}
\item $Lip(U) \leq a$, $U(y_0) = 0$ for some $y_0 \in B(y,R)$,
\item $\Delta_f U \leq b$ in the barrier sense, $U|_{\partial B(y,R)} \geq 0$.
\end{enumerate}
Then $U(y) \leq ac + b G_R(c)$ for all $0 < c < R$, where $G_R(r(x))$ is the smallest function on the model space $M_H^{n+4k}$ such that:
\begin{enumerate}
\item $G_R(r) > 0$, $G'_R(r) < 0$ for $0 < r < R$
\item $\Delta_H G_R \equiv 1$ and $G_R(R) = 0$.
\end{enumerate}
\end{prop}

\begin{proof}
Let $G_R(r)$ be the comparison function in the model space $M_H^{n+4k}$ as given in the statement of the theorem.  By the $f$-Laplacian Comparison, one has
$$ \Delta_f G_R \geq \Delta_H^{n+4k} G_R = 1.$$
Consider the function $V = b G_R - U$.  Then 
$$ \Delta_f V = b \Delta_f G_R - \Delta_f U \geq b \Delta_H^{n+4k} G_R - \Delta_f U = b - \Delta_f U \geq 0. $$
Then the maximal principle on $V: \overline{A(y,c,R)} \to \R$ gives
$$ V(x) \leq \max\{V|_{\partial B(y,R)}, V|_{\partial B(y,c)}\}$$
for all $x \in \overline{A(y,c,R)}.$ 
By assumption, we have 
$$V|_{\partial B(y,R)} = b G_R|_{\partial B(y,R)} - U|_{\partial B(y,R)} \leq 0$$
and 
$$V(y_0) = b G_R(y_0) - U(y_0) = b G_R(y_0) > 0.$$
Then there are two cases.  

If $y_0 \in A(y,c,R)$, then $\max_V|_{\partial B(y,c)} > 0$ so $V(y') > 0$ for some $y' \in \partial B(y,c)$.  Since 
$$ U(y) - U(y') \leq a \cdot d(y,y') = ac$$
and 
$$ b G_R(c) - U(y') = V(y') > 0,$$
it follows that 
$$ U(y) \leq ac + U(y') \leq ac + b G_R(c).$$
On the other hand, if $d(y, y_0) \leq c$, we may use the Lipschitz condition directly:
$$ U(y) = U(y) - U(y_0) \leq a \cdot d(y,y_0) \leq ac \leq ac + b G_R(c).$$
In either case, we have $U(y) \leq ac + b G_R(c)$ for all $0 < c < R$, as desired.  
\end{proof}

For any point $x \in M$, the excess function at $x$ is given by
\begin{equation} \label{excess}
e(x) = d(x, q_+) + d(x,q_-) - d(q_+, q_-).
\end{equation}
where $q_+, q_- \in M$ are fixed.  The Abresch-Gromoll Inequality for Riemannian manifolds gives an upper bound on the excess function in terms of a function
\begin{equation}\label{Psi}
\Psi = \Psi(\epsilon_1, \dots, \epsilon_k | c_1, \dots, c_N)
\end{equation}
such that $\Psi \geq 0$ and for any fixed $c_1, \dots, c_N$,
$$\lim_{\epsilon_1, \dots, \epsilon_k \to 0} \Psi = 0.$$
Such $\epsilon_i, c_i$ will be given explicitly below.  The proof of the Abresch-Gromoll Inequality for smooth metric measure spaces runs parallel to the proof one finds in \cite[Theorem 9.1]{Ch}, with the modification that one uses the 
Quantitative Maximal Principal \ref{QuantMaxP} together with the $f$-Laplacian Comparison \ref{fLapComp} in place of their Riemannian counterparts.  

\begin{prop}\label{AGIneq}{\bf (Abresch-Gromoll Inequality)}
 Given $R > 0$, $L > 2R + 1$ and $\epsilon > 0$, for any $p, q_+, q_- \in M^n$, if \eqref{hyp1} - \eqref{hyp2} hold, then
$$ e(x) \leq \Psi(H, L^{-1}, \epsilon | n, k, R)$$
on $B(p,R).$
\end{prop}

We note that an excess estimate for smooth metric measure spaces with $\Ric_f \geq 0$ and $|f| \leq k$ is given in  \cite[Theorem 6.1]{WW}.

For fixed $p, q_+, q_- \in M$, define the function $b_\pm: M \to \R$ by
$$b_\pm(x) = d(x, q_\pm) - d(p, q_\pm).$$
Let  $\mathbf{b}_\pm: M \to \R$ be the function such that 
\begin{equation}\label{b}
\Delta_f \mathbf{b}_\pm = 0\hspace{5 mm} \text{ and }\hspace{5 mm}b_\pm\big|_{\partial B(p,R)} = \mathbf{b}_\pm \big|_{\partial B(p,R)}.
\end{equation}
  The Abresch-Gromoll Inequality \ref{AGIneq} along with the $f$-Laplacian Comparison \ref{fLapComp} and the Maximal Principle \ref{QuantMaxP}, allow one to follow the proof of \cite[Lemma 6.15]{CC} to obtain the following.

\begin{lem}\label{PointwiseClose}   Given $R > 0$, $L > 2R + 1$ and $\epsilon > 0$, for any $p, q_+, q_- \in M^n$, if \eqref{hyp1} - \eqref{hyp2} hold
then on $B(p, R)$, 
\begin{equation}\label{Pointwise}
 |b_{\pm} - \mathbf{b}_{\pm} | \leq \Psi(H,L^{-1}, \epsilon|n,k,R).
 \end{equation}
\end{lem}

This pointwise estimate, along with the Gradient Estimate, Theorem \ref{GradEst}, and integration by parts then give

\begin{lem}\label{GradL2Close}
 Given $R > 0$, $L > 2R + 1$ and $\epsilon > 0$, for any $p, q+, q- \in M^n$, if \eqref{hyp1} - \eqref{hyp2} hold
\begin{equation}\label{Grad Close}
\dashint_{B(p,R)} |\nabla b_\pm - \nabla \mathbf{b}_\pm|^2 e^{-f} dvol \leq \Psi(H, L^{-1}, \epsilon| n, k, R)
\end{equation}
\end{lem}

Lemmas \ref{PointwiseClose} and \ref{GradL2Close} now allow one to obtain the key estimate for $\Hess \mathbf{b}_\pm$. The proof of the Hessian estimate is retained below for completeness.

As mentioned in the introduction, in order to obtain the Hessian estimate, one will multiply by a cutoff function $\phi$.  To construct such a function, one begins with a function $h: A(p, \frac{R}{2}, R) \to \R$ such that $\Delta_f h \equiv 1$, $h|_{\partial B(p, \frac{R}{2})} = G_R(R/2)$, $h|_{\partial B(p, R)} = 0$, where $G_R$ is as specified in Proposition \ref{QuantMaxP}.  Then let $\psi: [0, G_R(R/2)] \to [0,1]$ such that 
$\psi$ is 1 near $G_R(R/2)$ and $\psi$ is 0 near 0.  The function $\phi = \psi(h)$ extended to all of $M$ by setting $\phi = 1$ inside $B(p, R/2)$ and $\phi = 0$ outside of $B(p, R)$, is the cutoff function desired.  The gradient estimate \ref{GradEst2} guarantees that $|\grad \phi|$ and $|\Delta_f \phi|$ are bounded away from the boundary of the annulus on which $h$ was originally defined.  

\begin{proof}[Proof of Theorem \ref{HessEst}]
Applying the Bochner formula \eqref{Bochner} to the $f$-harmonic function $\ob$ yields
$$\frac{1}{2} \Delta_f |\nabla \mathbf{b}_\pm|^2 = | \Hess \ob|^2 + \Ric_f(\nabla \ob, \nabla \ob).$$
Multiply by a cutoff function $\phi$ that has the following properties:
\begin{itemize}
\item  $\phi \big|_{B(p, \frac{R}{2})} \equiv 1$,
\item $supp(\phi) \subset B(p, R)$,
\item $|\nabla \phi| \leq C(n, H, R,k)$,
\item $|\Delta_f \phi| \leq C(n, H, R,k)$.  
\end{itemize}
Then the above equation may be rewritten as
$$\phi | \Hess \ob|^2 = \frac{1}{2} \phi \Delta_f |\nabla \mathbf{b}_\pm|^2 - \phi \Ric_f(\nabla \ob, \nabla \ob)$$
Integrating both sides of this equality over $B(p,R)$ gives
\begin{align*}
\int_{B(p,R)} \phi | \Hess \ob|^2 e^{-f} dvol_g & =\int_{B(p,R)} \left( \frac{1}{2} \phi \Delta_f |\nabla \mathbf{b}_\pm|^2 - \phi \Ric_f(\nabla \ob, \nabla \ob)\right)e^{-f} dvol_g
\\ & \leq \int_{B(p,R)} \left( \frac{1}{2} \phi \Delta_f |\nabla \mathbf{b}_\pm|^2 + (n-1)H|\nabla \ob|^2\right)e^{-f} dvol_g
\\ & =  \frac{1}{2} \int_{B(p,R)} \phi \Delta_f |\nabla \mathbf{b}_\pm|^2 e^{-f} dvol_g 
\\ & \hspace{ 20 mm}+ \int_{B(p,R)}\  (n-1)H|\nabla \ob|^2 e^{-f} dvol_g
\end{align*}
For the first integrand, we have
\begin{align*}
%1
 \int_{B(p,R)} \phi \Delta_f |\nabla \mathbf{b}_\pm|^2 e^{-f} dvol_g  &=  \int_{B(p,R)} \phi \Delta_f( |\nabla \mathbf{b}_\pm|^2-1) e^{-f} dvol_g 
\\ & = \int_{B(p,R)} \Delta_f \phi  (|\nabla \ob|^2 -1)  e^{-f} dvol_g
\end{align*}
Thus 
\begin{align*}
\int_{B(p,R)} \phi | \Hess \ob|^2 e^{-f} dvol_g &\leq \left[ \int_{B(p,R)} \frac{1}{2} \Delta_f \phi (|\nabla \ob|^2-1) 
  +  (n-1)H|\nabla \ob |^2\right] e^{-f} dvol_g
  \\ & \leq \int_{B(p,R)} \left[\frac{1}{2} |\Delta_f \phi| ||\nabla \ob|^2-1|  +  (n-1)H|\nabla \ob |^2 \right] e^{-f} dvol_g
\end{align*}
Since
$$ ||\nabla \ob|^2 - 1| = ||\nabla \ob| - |\nabla b_\pm||(|\nabla \ob| + 1) \leq |\nabla \ob - \nabla b_\pm|(|\nabla \ob|+1),$$
we have 
$$ \dashint_{B(p,R)} \phi | \Hess \ob|^2 e^{-f} dvol_g \leq \Psi.$$
\end{proof}

This Hessian estimate is important because it, together with the Segment Inequality, Proposition \ref{SegmentInequality}, allow us to extend the Quantitative Pythagorean Theorem to the smooth metric measure space setting as follows.

\begin{prop}\label{PythThm}{\bf (Quantitative Pythagorean Theorem)}
 Given $R > 0$, $L > 2R + 1$ and $\epsilon > 0$, for any $p, q+, q- \in M^n$, assume \eqref{hyp1} - \eqref{hyp2} hold.
Let $x, z, w \in B(p, \frac{R}{8})$, with $x \in \mathbf{b}_+^{-1}(a)$, and $z$ a point on $\mathbf{b}_+^{-1}(a)$ closest to $w$.  Then
$|d(x,z)^2 + d(z,w)^2 - d(x,2)^2| \leq \Psi$.
\end{prop}

From this Quantitative Pythagorean Theorem for smooth metric measure spaces,  one may establish the following Almost Splitting Theorem.

\begin{thm}\label{Almost Splitting}{\bf (Almost Splitting Theorem)}
Given $R > 0$, $L > 2R + 1$ and $\epsilon >0$, let $p, q_+, q_- \in M^n$.  If $(M^n, g, e^{-f}\dvol_g)$ satisfies \eqref{hyp1} - \eqref{hyp2}, then there is a length space $X$ such that for some ball $B((0,x), \frac{R}{4}) \subset \R \times X$ with the product metric, we have
$$d_{GH}(B(p, \frac{R}{4}), B((0,x),\frac{R}{4})) \leq \Psi(H, L^{-1}, \epsilon | k, n, R).$$
\end{thm}

From this Almost Splitting Theorem for smooth metric measure spaces, it follows that the splitting theorem extends to the limit of a sequence of smooth metric measure spaces is the following manner.

\begin{thm}\label{Splitting}
Let $(M^n_i, g_i, e^{-f_i}dvol_{g_i})$ be a sequence satisfying the following:  $M_i^n \to Y$, $\Ric_{f_i}M_i \geq -(n-1)\delta_i$, where $\delta_i \to 0$, $|f_i| \leq k$.  If $Y$ contains a line, then $Y$ splits as an isometric product, $Y = \R \times X$ for some length space $X$.  
\end{thm}

   \section{Polynomial Growth of the Fundamental Group}
   
   As mentioned in the introduction, the first theorem which we wish to extend to smooth metric measure spaces, Yun's theorem \cite[Main Theorem]{Yun}, is actually a strengthening of Wei's theorem \cite{Wei}. 
   
   \begin{thm}\cite[Theorem 1]{Wei}\label{WeiPoly}
For any constant $v > 0$, there exists $\epsilon = \epsilon(n, v) > 0$ such that if a complete manifold $M^n$ admits a metric satisfying the conditions 
$
\Ric \geq -\epsilon,
$
$
 d(M) = 1,
$ and
$
  \vol(M) \geq v,
$ then the fundamental group of $M$ is of polynomial growth with degree $\leq n$.
   \end{thm} 
   
Yun uses the existence of such $\epsilon$ to construct a contradicting sequence of Riemannian manifolds $M_i$ such that $\Ric(M_i) \geq -\epsilon_i \to 0$, where $\epsilon_i \leq \epsilon$, $\vol(M_i) \geq v$, and $\diam(M_i) \leq D$ but $\pi_1(M_i)$ is not almost abelian.  It is with this sequence that Yun utilizes the Almost Splitting Theorem.  If we wish to generalize his arguments to the smooth metric measure space setting, we should also establish the existence of such an $\epsilon$ for smooth metric measure spaces.  
    
Wei's proof of Theorem \ref{WeiPoly} requires use of the Bishop-Gromov absolute volume comparison.  The relative volume comparison on smooth metric measure spaces formulated  in \cite[Theorem 1.2b]{WW}  only yields a volume growth estimate for $R >1$ since, as noted by Wei and Wylie, the right hand side of 
$$\frac{\vol_f(B(p,R))}{V_H^{n+4k}(R)}\leq \frac{\vol_f(B(p,r))}{V_H^{n+4k}(r)}$$
blows up as $r \to 0$.  Using this type of estimate to extend Wei's proof methods to the smooth metric measure space setting would require additional assumptions.  Moreover, using this comparison will yield polynomial growth of degree $n + 4k$.  In order to improve the degree with only the additional assumption that $|f|\leq k$, we formulate the following volume estimate.

\begin{prop}\label{AbsVolComp}
Let $(M^n, g, e^{-f} dvol_g)$ be a smooth metric measure space with $\Ric_f \geq (n-1)H$, $H < 0$, and $|f| \leq k$.  Then 
$$Vol_f(B(R_1)) \leq k \int_0^{R_2} \mathcal{A}_H(r)e^{2k[\cosh(2\sqrt{-H}r)+1]} dr,$$
where $\mathcal{A}_H(r)dr$ denotes the volume element on the model space with constant curvature $H$.  
\end{prop}

\begin{proof}

Let $sn_H(r)$ be the solution to $sn_H'' + H sn_H = 0$
such that $sn_H(0) = 0$ and $sn'_H(0)= 1$.  When $H < 0$, this solution is given by 
\begin{equation}\label{soln}
\frac{1}{\sqrt{-H}} \sinh \sqrt{-H}r.
\end{equation}  
From the proof \cite[Theorem 1.1]{WW}, inequality (2.17) gives
\begin{equation}\label{2.17}
sn_H^2(r)m_f(r) \leq sn_H^2(r)m_H(r) - f(r) (sn_H^2(r))' + \int_0^r f(t)(sn_H^2)''(t)dt.
\end{equation}
Then integrating both sides of \eqref{2.17} from $r=r_1$ to $r_2$ gives
\begin{align*}
\int_{r_1}^{r_2}m_f(r)dr &\leq \int_{r_1}^{r_2}m_H(r)dr - \int_{r_1}^{r_2}f(r) \frac{(sn_H^2(r))'}{sn_H^2(r)}dr + \int_{r_1}^{r_2}\frac{1}{sn_H^2(r)}\left\{ \int_0^r f(t)(sn_H^2)''(t)dt\right\}dr
\\ & = \int_{r_1}^{r_2} m_H(r)dr -  2\sqrt{-H}\int_{r_1}^{r_2}f(r) \coth \sqrt{-H}r dr 
\\ & +2(-H)\int_{r_1}^{r_2}  \csch^2\sqrt{-H}r \left\{ \int_0^r f(t) [\sinh^2 \sqrt{-H} t + \cosh^2 \sqrt{-H}t]dt\right\} dr
\\ & = \int_{r_1}^{r_2} m_H(r)dr -  2\sqrt{-H}\int_{r_1}^{r_2}f(r) \coth \sqrt{-H}r dr 
\\ & +2(-H)\int_{r_1}^{r_2}  \csch^2\sqrt{-H}r \left\{ \int_0^r f(t) \cosh 2\sqrt{-H}t dt\right\} dr
\\& = \int_{r_1}^{r_2} m_H(r)dr -  2\sqrt{-H}\int_{r_1}^{r_2}f(r) \coth \sqrt{-H}r dr 
\\&+ 2(-H)\left[-\frac{\coth \sqrt{-H}r}{\sqrt{-H}}\int_0^r f(t) \cosh 2\sqrt{-H}t dt \right]_{r_1}^{r_2} 
 \\&+  4(-H)\int_{r_1}^{r_2} \frac{\coth \sqrt{-H}r}{\sqrt{-H}}f(r) \sinh^2\sqrt{-H}rdr 
 + 2(-H)\int_{r_1}^{r_2} \frac{\coth \sqrt{-H}r}{\sqrt{-H}}f(r) dr
 \\ &\leq  \int_{r_1}^{r_2} m_H(r)dr + k \coth \sqrt{-H}r_2 \sinh 2\sqrt{-H}r_2 + k \coth \sqrt{-H}r_1\sinh 2\sqrt{-H}r_1 
 \\ &+ 2k[\sinh^2\sqrt{-H}r_2 - \sinh^2\sqrt{-H}r_1]
\\ &  =  \int_{r_1}^{r_2} m_H(r)dr + 2k[\cosh(2\sqrt{-H} r_2) + 1].
\end{align*}
where the first equality is obtained by substituting \eqref{soln} for $sn_H$, and the third equality is obtained through integration by parts.  

Using exponential polar coordinates around $p$, we may write the volume element of $M$ as $\mathcal A(r, \theta) \wedge d\theta_{n-1}$ where $d\theta_{n-1}$ is the standard volume element of the unit sphere $\mathbb{S}^{n-1}$.  Let $\mathcal{A}_f (r, \theta)= e^{-f} \mathcal{A}(r, \theta)$ and $\mathcal{A}_H(r)$ denotes the volume element for the model space with constant curvature $H$.  The mean curvatures on the smooth metric measure space and on the model space can be written, respectively, as
$$m_f(r) = (\ln(\mathcal A_f(r,\theta))' \hspace{5 mm}\text{and} \hspace{5 mm} m_H(r) = (\ln(\mathcal A_H(r))'.$$
Then we may rewrite the above inequality as
$$\ln\bigg(\frac{\mathcal A_f(r_2,\theta)}{\mathcal A_f(r_1,\theta)}\bigg) \leq \ln\bigg(\frac{\mathcal A_H(r_2)}{\mathcal A_H(r_1)}\bigg) + 2k[\cosh(2\sqrt{-H} r_2) + 1].$$
Hence
$$\frac{\mathcal A_f(r_2,\theta)}{\mathcal A_f(r_1,\theta)} \leq  \frac{\mathcal A_H(r_2)}{\mathcal A_H(r_1)}e^{2k[\cosh(2\sqrt{-H} r_2) + 1].}$$
Then
$$\mathcal A_f(r_2,\theta){\mathcal A_H(r_1)} \leq  {\mathcal A_H(r_2)}{\mathcal A_f(r_1,\theta)}e^{2k[\cosh(2\sqrt{-H} r_2) + 1]}.$$
Integrating both sides of the inequality over $\mathbb S^{n-1}$ with respect to $\theta$ yields
$$ \mathcal A_H(r_1) \int_{\mathbb S^{n-1}} \mathcal A_f(r_2, \theta)d\theta \leq  \mathcal A_H(r_2) e^{2k[\cosh(2\sqrt{-H} r_2) + 1]}\int_{\mathbb S^{n-1}} \mathcal A_f(r_1, \theta)d\theta.$$
Then we integrate both sides of the inequality with respect to $r_1$ from $r_1 = 0$ to $r_1 = R_1$:
$$ Vol_H(B(R_1)) 
\int_{\mathbb S^{n-1}}\mathcal A_f(r_2, \theta)d\theta \leq Vol_f(B(R_1))  \mathcal A_H(r_2) e^{2k[\cosh(2\sqrt{-H} r_2) + 1]}.$$
Finally, we integrate  both sides of the inequality with respect to $r_2$ from $r_2 = 0$ to $r_2 = R_2$:
$$ Vol_H(B(R_1)) 
Vol_f(B(R_2)) \leq Vol_f(B(R_1)) \int_0^{R_2}  \mathcal A_H(r_2) e^{2k[\cosh(2\sqrt{-H} r_2) + 1]}dr_2,$$
thus yielding a new volume inequality:
$$ \frac{Vol_H(B(R_1)) }{Vol_f(B(R_1)}
\leq 
\frac{\int_0^{R_2}  \mathcal A_H(r_2) e^{2k[\cosh(2\sqrt{-H} r_2) + 1]}dr}{Vol_f(B(p,R_2))}. $$
Note that the left hand side of the inequality tends to $\frac{1}{f(p)}$ as $R_1 \to 0$.  Then
$$Vol_f(B(p, R_2)) \leq f(p) \int_0^{R_2}  \mathcal A_H(r_2) e^{2k[\cosh(2\sqrt{-H} r_2) + 1]}dr_2$$
\end{proof}

Using Prop \ref{AbsVolComp}, we may extend some of the results of Michael Anderson from \cite{And} to the smooth metric measure space setting.  These are also stated in \cite{WW} without proof.

\begin{prop}
\label{Length SMMS}
Let $(M^n, g, e^{-f}dvol_g)$ be a smooth metric measure space with $|f| \leq k$ satisfying the bounds
 $\Ric_f \geq (n-1) H$, $\diam_M \leq D$ and $\vol_f(M) \geq v$.   
If $\gamma$ is a curve in $M$ with $[\gamma]^p \neq 0$ in $\pi_1(M)$ for all $p \leq N= \frac{k}{v} \int_0^{2D} \mathcal{A}_H e^{2k[\cosh(2\sqrt{-H}r)+1]} dr$, then 
$$ l(\gamma) \geq \frac{D}{N}.$$
\end{prop}
\vspace{3 mm}

Proposition \ref{Length SMMS} is used directly in the proof of Theorem \ref{AlmostAbelian}.  It is also used to prove the following extension of another theorem of Anderson.

\begin{prop}\label{IsomorphismTypes}
For the class of manifolds $M^n$ with $\Ric_f \geq (n-1) H$, $\vol_f \geq v$, $\diam_M \leq D$ and $|f| \leq k$, there are only finitely many isomorphism types of $\pi_1(M)$.  
\end{prop}

Now, with Proposition \ref{AbsVolComp} and Proposition \ref{IsomorphismTypes}, we may extend Wei's theorem about polynomial growth of the fundamental group \cite{Wei} to smooth metric measure spaces.

\begin{thm}\label{SMMS Wei}
For any constant $v \geq 0$, there exists $\epsilon = \epsilon(n, v, k, H, D) > 0$ such that if a smooth metric measure space $(M^n, g, e^{-f}dvol_g)$ with $|f| \leq k$ satisfies the conditions \eqref{Ric} - \eqref{vol}, then the fundamental group of $M$ is of polynomial growth of degree $\leq n$.  
\end{thm}

\begin{proof}
Let us assume by means of contradiction that $\pi_1(M)$ is not of polynomial growth with degree $\leq n$. It follows that for any set of generators of $\pi_1(M)$, we can find real numbers $s_i$ for all $i$, such that 
\begin{equation}\label{contradiction}
\gamma(s_i) > is_i^{n}.
\end{equation}

Choose a base point $\tilde x_0$ in the universal covering $p: \tilde M \to M$, and let $x_0 = p(\tilde x_0)$.
By Proposition \ref{IsomorphismTypes}, there are only finitely many isomorphism types of $\pi_1(M)$.  For each isomorphism type, choose a set of generators $g_1, \dots, g_N$ such that $d(g_i(\tilde x_0), \tilde x_0) \leq 3D$ and every relation is of the form $g_i g_j = g_k$.  Such a set of generators is guaranteed by a theorem of Gromov \cite{Gro}.  By the proof of Proposition \ref{IsomorphismTypes}, we know that $N$ is uniformly bounded.  Having chosen generators in this manner, we are guaranteed that \eqref{contradiction} is independent of $\epsilon$.  
View this set of generators of the fundamental group $\pi_1(M)$ as deck transformations in the isometry group of $\tilde M$.  Let 
$\Gamma(s) = \{ \text{distinct words in } \pi_1(M) \text{ of length } \leq s\},$
and $\gamma(s) = \# \Gamma(s)$.

Now, choose a fundamental domain $F$ of $\pi_1(M)$ containing $\tilde x_0$.  Then 
$$\bigcup_{g \in \Gamma(s)} g(F) \subseteq B(\tilde x_0,D(3s+1) )),$$
which implies 
$$ \gamma(s) \leq \frac{1}{v} \vol_f(B_(\tilde x_0,D(3s+1) )).$$
Then, by Proposition \ref{AbsVolComp}, it follows that 
$$ \gamma(s) \leq \frac{k}{v} \int_0^{D(3s+1)}   \frac{\sinh \sqrt{\epsilon}r}{\epsilon} e^{2k[\cosh(2\sqrt{\epsilon} r) + 1]}dr.$$
For any fixed, sufficiently large $s_0$, there exists $\epsilon_0 = \epsilon(s_0)$ such that for all $\epsilon \leq \epsilon_0$, we have
\begin{equation}\label{contradiction2}
\gamma(s) \leq \frac{2^{3n} e^4 k}{nv} s^n.
\end{equation}
Let $i_0 > \frac{2^{3n} e^4 k}{nv}$ .  Then $\epsilon < \epsilon(s_{i_0})$ together with \eqref{contradiction} and \eqref{contradiction2} yields a contradiction.  
\end{proof}

\section{Proof of Theorem \ref{AlmostAbelian}}

With Theorems \ref{Splitting}, \ref{SMMS Wei}, and Proposition \ref{Length SMMS}, one may generalize the arguments in \cite{Yun} to the smooth metric measure space setting.  For completeness, we retain the proof here.    

   \begin{proof}[Proof of Theorem \ref{AlmostAbelian}]
By Theorem \ref{SMMS Wei}, there exists $\epsilon_0 = \epsilon_0(n, v, k, H, D) > 0$ such that if a smooth metric measure space $(M, g, e^{-f}\dvol_g)$ with $|f|< k$, satisfies \eqref{Ric} - \eqref{vol}, then $\pi_1(M)$ is a finitely generated group of polynomial growth of order $\leq n$.  

Assume Theorem \ref{AlmostAbelian} is not true.  Then there exists a contradicting sequence of smooth metric measure spaces $(M_i, g_i, e^{-f_i} \dvol_{g_i})$ with $|f_i| \leq k$ and
   $$\Ric_{f_i}(M_i) \geq -\epsilon_i \to 0, \hspace{3 mm} \epsilon_i \leq \epsilon_0, \hspace{3 mm} \vol_{f_i}(M_i) \geq v, \hspace{3 mm} \diam(M_i) \leq D,$$
   such that $\pi_1(M_i)$ is not almost abelian for each $i$.  Note, however, $\pi_1(M_i)$ is of polynomial growth for each $i$.
   
   Since $\pi_1(M_i)$ is of polynomial growth, \cite[Lemma 1.3]{Yun} implies it contains a torsion free nilpotent subgroup $\Gamma_i$ of finite index.   Since $\Gamma_i$ has finite index in $\pi_1(M)$, it must be nontrivial.  Furthermore $\Gamma_i$ cannot be almost abelian.  
   
   Consider the action $\Gamma_i$ on the universal cover $\widetilde M_i$.  For $p_i \in \widetilde M_i$ consider the sequence $(\widetilde M_i, \Gamma_i, p_i)$.  There exists a length space $(Y, q)$ and a closed subgroup $G$ of $Isom(Y)$ such that $(\widetilde M_i, \Gamma_i, p_i)$ subconverges to a triple $(Y, G, q)$ with respect to the pointed equivariant Gromov Hausdorff distance \cite[Theorem 3.6]{FY}.
  
Using the Almost Splitting Theorem \ref{Splitting}, we know $Y$ splits as an isometric product $Y = \R^k \times Y_0$ for some $k$ and length space $Y_0$ which contains no lines.  By Proposition \ref{IsomorphismTypes} it follows  $[\pi_1(M_i): \Gamma_i]$ is uniformly bounded, say   $[\pi_1(M_i): \Gamma_i]\leq m$.  Hence $\diam(\widetilde M_i/\Gamma_i) \leq Dm$.  Then since $(\widetilde{M_i}, \gamma_i, p_i) \to (\R^k \times Y_0, G, q)$, it follows that $\diam(\R^k \times Y_0/G) \leq Dm$. Then $Y_0$ must be compact.  Otherwise, it would contain a line. 
Thus we may consider the projection
$$ \phi: G \to Isom(\R^k).$$
  By \cite[Theorem 6.1]{FY}, for every $\delta > 0$ there exists a normal subgroup $G_{\delta}$ of $G$ such that
$G/ G_{\delta}$ contains a finite index, free abelian group of rank not greater than $\dim(\R^k/ \phi(G))$.  Since $\Gamma_i$ is torsion free, Proposition \ref{Length SMMS} gives that for all nontrivial $\gamma \in \Gamma_i$,   we have $l(\gamma) \geq \frac{D}{N}$ where $N =  \frac{k}{v} \int_0^{2D} \mathcal{A}_{\epsilon_0} e^{2k[\cosh(2\sqrt{\epsilon_0}r)+1]} dr$.  Choose $\delta = \frac{D}{N}$ and set $\delta_0 = \delta/2$.  

 Define
    $$\Gamma_i (\delta) = \{\gamma \in \Gamma_i: d(p_i, \gamma(p_i)) < \delta \}.$$
Similarly, define
     $$G(\delta) = \{ \gamma \in G: d(q, \gamma(q)) < \delta\}.$$
Then 
  $$\Gamma_i(\delta) = \{1\}.$$
Since $(\widetilde{M_i}, \gamma_i, p_i) \to (\R^k \times Y_0, G, q)$, it follows that 
$$G(\delta_0)= \{1\}.$$

Let $K$ denote the kernel of $\phi$.  Since $\delta_0 > 0$ was chosen so that $G(\delta_0) = \{1\}$, it follows that 
$$ \{\gamma \in K | d(\gamma(x), x) < \delta_0 \text{ for all } x \in Y\} = \{1\}.$$
Thus the subgroup generated by this set is trivial.  That is, 
$$K_{\delta_0} = \langle \{\gamma \in K | d(\gamma(x), x) < \delta \text{ for all } x \in Y\}\rangle = \{1\}.$$
Then, the quotient map
$$ \pi: G \to G/{K_{\delta_0}}$$
is simply the identity map.  
The subgroup $G_{\delta_0}$ of $G$ which has the properties we seek is defined by
$$G_{\delta_0} = \pi^{-1}([1]),$$
where $[1]$ denotes the coset containing the identity element of $G/K_{\delta_0}$.
But since $K_{\delta_0}$ is trivial and $\pi$ is the identity map, it follows that $G_{\delta_0} = \{1\}$.
Thus by \cite[Lemma 6.1]{FY}, $G/G_{\delta_0} = G$ contains a finite index free abelian subgroup of rank $\leq k$; that is, $G$ is almost abelian.  %%%%%%%%%%%
Moreover,  \cite[Theorem 3.10]{FY}
 %%%%%%%%%%%%%%%%%%%%%%%%%%%%%%%%%%%%%
we have that $\Gamma_i$ is isomorphic to $G$ for $i$ sufficiently large.  But this contradicts the fact that $\Gamma_i$ is not almost abelian for each $i$.
\end{proof}

\section{Bound on number of generators of the fundamental group}

In order to obtain a uniform bound on the number of generators of the fundamental group, Kapovitch and Wilking require two results closely related to the Almost Splitting Theorem.  The first of these results is due to Cheeger and Colding \cite{CC2}, see also \cite{CC}. %and \cite[Theorem 9.29]{Ch}.  

\begin{thm}\label{9.29}  Given $R >0$ and $L > 2R + 1$,
let $\Ric_{M^n} \geq -(n-1)\delta$ and $d_{GH} (B(p, L), B(0,L)) \leq \delta$, where $B(0,L) \subset \R^n$.  Then there exist harmonic functions $\mathbf{b}_1, \dots, \mathbf{b}_n$ on $B(p, R)$ such that in the Gromov-Hausdorff sense
$d(e_i, \mathbf{b}_i) \leq \Psi,$
where $\{e_i\}$ denote the standard coordinate functions on $\R^n$ and
\begin{equation}\label{IntIneq}
\dashint_{B(p,R)} \sum_i |\nabla \mathbf{b}_i - 1|^2 + \sum_{i \neq j} |\langle \nabla \mathbf{b}_i, \nabla \mathbf{b}_j\rangle| + \sum_i |\Hess \mathbf{b}_i|^2 \leq \Psi
\end{equation}
\end{thm}

 In the smooth metric measure space setting, a similar statement may be made:

\begin{thm}\label{9.29 SMMS}
Let $\Ric_{f} \geq -(n-1)\delta$, with $|f| \leq k$ and $d_{GH} (B(p,L), B(0,L)) \leq \delta$, where $B_L(0) \subset \R^n$.  Then there exist $f$-harmonic functions $\mathbf{b}_1, \dots, \mathbf{b}_n$ on $B_R(p)$ such that in the Gromov-Hausdorff sense
$d(e_i, \mathbf{b}_i) \leq \Psi,$
where $\{e_i\}$ denote the standard coordinate functions on $\R^n$ and
\begin{equation}\label{IntIneq}
\dashint_{B(p,R)} \left( \sum_i |\nabla \mathbf{b}_i - 1|^2 + \sum_{i \neq j} |\langle \nabla \mathbf{b}_i, \nabla \mathbf{b}_j\rangle| + \sum_i |\Hess \mathbf{b}_i|^2\right)e^{-f}\dvol_g \leq \Psi
\end{equation}
\end{thm}

The manner in which the harmonic functions $\mathbf{b}_i$ are constructed for Theorems \ref{9.29} and \ref{9.29 SMMS} is similar to the manner in which the harmonic functions $\mathbf{b}_\pm$ are constructed in the proof of the Almost Splitting Theorem in both the Riemannian and smooth metric measure space settings.  Since the two $L$-balls are $\delta$-close in the Gromov-Hausdorff sense, there exists a $\delta$-Gromov-Hausdorff approximation
$$F: B(0,L) \to B(p, L).$$
For each $i = 1, \dots, n$, set
$$q_i = F(Le_i)$$
and define $b_i: M \to \R$ by
$$ b_i(x) = d(x, q_i) - d(p, q_i).$$
For the smooth metric measure space version, let $\mathbf{b}_i$ be the $f$-harmonic function such that $\mathbf{b}_i \big|_{\partial B(p, L)} = b_i \big|_{\partial B(p, L)}$.  
Integrating each term separately, we see that 
the first term can be controlled by \eqref{Grad Close} and the third term by \eqref{Hess}.  One can show a $\Psi$-upper bound for the middle term of the integrand \eqref{IntIneq} by noting that
\begin{align*}
\langle \nabla \mathbf{b}_i, \nabla \mathbf{b}_j \rangle & = \langle \nabla \mathbf{b}_i - \nabla b_i + \nabla b_i ,  \nabla \mathbf{b}_j  -\nabla b_j  +\nabla b_j \rangle 
\\ & = \langle \nabla \mathbf{b_i} - \nabla b_i, \nabla \mathbf b_j \rangle  + \langle \nabla \mathbf{b_i} - \nabla b_i, \nabla  b_j \rangle + \langle \nabla \mathbf{b}_j - \nabla b_j, \nabla \mathbf b_i\rangle 
\\ & \hspace{20 mm} + \langle \nabla \mathbf{b_j} - \nabla b_j, \nabla b_i \rangle + \langle \nabla b_i,  \nabla b_j \rangle
\end{align*}
Using integration by parts and \eqref{Pointwise}, one can show that the average value of each of the first four terms of the summand is bounded from above by $\Psi$.  Moreover, $\langle \grad \mathbf{b_i}, \grad \mathbf{b_j} \rangle \to 0$ when $L \to \infty$.

The Product Lemma  of Kapovitch and Wilking, stated below, can be viewed as another type of splitting result.

\begin{thm} \label{ProductLemma}\cite[Lemma 2.1]{KW}  Let $M_i$ be a sequence of manifolds with $\Ric_{M_i} > -\epsilon_i \to 0$ satisfying
\begin{itemize}
\item $
\overline{B_{r_i}(p_i)}$ compact for all $r_i \to \infty$, $p_i \in M_i$, 
\item  for all $i$ and $j = 1, \dots, k$ there exist harmonic functions $\mathbf{b}_j^i: B(p_i, r_i) \to \R$  which are $L$-Lipschitz and fulfill 
$$ \dashint_{B(p_i, R)} \left(\sum_{j,l = 1}^k |\langle \nabla \mathbf{b}_j^i, \nabla b_l^i \rangle - \delta_{jl}| + \sum_{j = 1}^k |\Hess \mathbf{b}_j^i|^2\right) d\mu_i \to 0 \hspace{5 mm} \text{ for all } R > 0,$$
\end{itemize}
then the limit is isometric to a product $(\R^k \times X, p_\infty)$. 
\end{thm}

  Without the assumption that a line exists in the limit space, Kapovitch and Wilking instead show that each of the functions $\mathbf{b}_j^i$, as in the hypothesis of Theorem \ref{ProductLemma}, limit to a submetry $\mathbf{b}_j^\infty$ as $i \to \infty$.  The submetry $\mathbf{b}_j^\infty$  then lifts lines to lines, which allows one to apply the Almost Splitting Theorem to show that the limit indeed splits.  Their argument may be modified to the smooth metric measure space setting by using the volume comparison \cite[Theorem 1.2]{WW}, the Segment Inequality \ref{SegmentInequality}, and the fact that gradient flow of an $f$-harmonic function is measure preserving with respect to the measure $e^{-f} \dvol_g$.  Augmenting their arguments in this manner yields the following extension.

\begin{thm} \label{Product lemma SMMS}
Let $(M_i, g_i, e^{-f}\dvol_{g_i})$ be a sequence of smooth metric measure spaces with $|f_i| \leq k$ and $\Ric_{f_i} > -\epsilon_i \to 0$.  Suppose that for every $i$ and $j = 1, \dots, m$, there are harmonic functions $b_j^i: B(p_i, r_i) \to \R$ which are $L$-Lipschitz and fulfill
$$\dashint_{B(p_i, R)} \left(\sum_{j,l=1}^m |\langle \nabla \mathbf{b}_j^i, \nabla \mathbf{b}_l^i \rangle - \delta_{jl}| + \sum_{j=1}^m |\Hess \mathbf{b}_j^i|^2\right) e^{-f_i} \dvol_{g_i} \to 0 \text{ for all } R > 0.$$
Then $(B(p_i, r_i), p_i)$ subconverges in the pointed Gromov-Hausdorff topology to a metric product $(\R^m \times X, p_\infty)$ for some metric space $X$.
\end{thm}

The following lemma of Kapovitch and Wilking requires only an inner metric space structure and hence may be applied to smooth metric measure spaces.  

\begin{lem}\cite[Lemma 2.2]{KW} \label{2.2 SMMS}
Let $(Y_i, \tilde p_i)$ be an inner metric space endowed with an action of a closed subgroup $G_i$ of its isometry group, $i \in \mathbb{N} \cup \{\infty\}$.  Suppose $(Y_i, G_i, \tilde p_i) \to (Y_\infty, G_\infty, \tilde p_\infty)$ in the equivariant Gromov-Hausdorff topology.  Let $G_i(r)$ denote the subgroup generated by those elements that displace $\tilde p_i$ by at most $r$, $i \in \mathbf{N} \cup \{ \infty \}$.  Suppose there are $0 \leq a < b$ with $G_\infty(r) = G_\infty(\frac{a+b}{2})$ for all $r \in (a, b)$.  Then there is some sequence $\epsilon_i \to 0$ such that $G_i(r) = G_i(\frac{a+b}{2})$ for all $r \in (a+ \epsilon_i, b-\epsilon_i)$.  
\end{lem}

For more on equivariant Gromov-Hausdorff convergence, see \cite{FY}.  Lemma \ref{2.2 SMMS} and the Almost Splitting Theorem \ref{Splitting} allow us to modify arguments of the proof of \cite[Lemma 2.3]{KW}  to show that the following holds for smooth metric measure spaces.  

\begin{lem}\label{2.3 SMMS}
Suppose $(M_i^n, q_i)$ is a pointed sequence of smooth metric measure spaces where $(M_i^n, g_i, e^{-f_i}\dvol_{g_i})$  has $|f_i| \leq k$ and $\Ric_{f_i}(M_i)\geq -1/i$.  Moreover, assume $(M_i^n, q_i) \to (\R^m \times K, q_\infty)$ where $K$ is compact, and the action of $\pi_1(M_i)$ on the universal cover $(\widetilde M_i, \tilde q_i)$ converges to a limit action of a group $G$ on some limit space $(Y, \tilde q_\infty)$.  Then $G(r) = G(r')$ for all $r, r' > 2 \diam(K)$.
\end{lem}

\noindent
We will also need the following result on the dimension of the limit space. 

\begin{lem}\label{Limit Dimension}
Let $(M_i^n, g_i, e^{-f_i}\dvol_g)$ be a sequence of smooth metric measure spaces such that $|f_i| \leq k$, $\diam (M_i^n) \leq D$, and $\Ric_f \geq -(n-1)H$, $H > 0$.    If $M_i^n$ converges to the length space $Y^m$ in the Gromov-Hausdorff sense, then $m \leq n+4k$.  
\end{lem}

%%%%%%%%%%%%%%%%%%%%%%%%%%%%
\begin{proof}
Begin by noting that for any $(M^n, g, e^{-f}\dvol_g)$ with $\Ric_f \geq -(n-1)H$, $H > 0$, and fixed $x \in M$ and $R>0$, the $f$-volume comparison \cite[Theorem 1.2b]{WW} gives a bound on the number of disjoint $\epsilon$-balls contained in $B(x, R)$:
Let $B(x_1, \epsilon), \dots, B(x_l, \epsilon) \subset B(x, R)$ be disjoint.  Let $B(x_i, \epsilon)$ denote the ball with the smallest $f$-volume.  Then
$$ l \leq \frac{\vol_f B(x, R)}{\vol_f B(x_i, \epsilon)} \leq \frac{\vol_f B(x_i, 2R)}{\vol_f B(x_i, \epsilon)} \leq \frac{\vol_{H}^{n+4k}B(2R)}{\vol_{H}^{n+4k}B(\epsilon)}=C(n+4k, H, R, \epsilon).$$

Thus $\Capa_{M_i}(\epsilon)$, the maximum number of disjoint $\epsilon/2$-balls which can be contained in $M_i^n$, is bounded above by $C=C(n+4k, H, D, \frac{\epsilon}{2})$ for each $i$.  Moreover $\Cov_{M_i}(\epsilon)$, the   minimum number of $\epsilon$-balls covering $M_i^n$ less than or equal to $\Capa_{M_i}(\epsilon)$, so $\Cov_{M_i} \leq C$.    

Since $M_i^n \to Y$ in the Gromov-Hausdorff sense, there exists a sequence $\delta_i > 0$ such that $d_{GH}(M_i, Y) <  \delta_i \to 0$ as $i \to \infty$.  Then $\Cov_{Y}(\epsilon) \leq \Cov_{M_i}(\epsilon- 2\delta_i) \leq C$.  As $i \to \infty$, we have $\Cov_Y(\epsilon) \leq C$.  

To see that the Hausdorff dimension is bounded above by $n+4k$, recall that the $d$-dimension Hausdorff measure of $Y$ is defined by 
$$ H^d(Y) = \lim_{\epsilon \to 0} H_\epsilon^d (Y),$$
where 
$$H_\epsilon^d(Y) = \inf \left\{\sum_{i=1}^\infty (\diam U_i)^d \bigg| \bigcup_{i=1}^\infty U_i \supset Y, \diam U_i \leq \epsilon \right\}.$$
Since $\Cov_Y(\epsilon) \leq C$, it follows that $H_\epsilon^d(Y) \leq \sum_{i=1}^C (2\epsilon)^d$.  
Notice
$$ C = \frac{\vol_{H}^{n+4k}B(D)}{\vol_{H}^{n+4k}B(\epsilon/2)} \sim (\epsilon/2)^{-(n+4k)}$$
as $\epsilon \to 0$.  Thus as $\epsilon \to 0$
$$
 \sum_{i=1}^C (2\epsilon)^d
 = C\frac{\epsilon}{2}^d
\to 0
$$
for all $d > n+4k$.  
Thus the Hausdorff dimension of $Y$, defined by $\dim_H(Y) = \inf\{d \geq 0 | H^d(Y) = 0\}$ is at most $n+4k$.  
\end{proof}
%%%%%%%%%%%%%%%%%%%%%%%%%%%%

The final tool we will use to extend \cite[Theorem 2.5]{KW} to smooth metric measure spaces is a type of Hardy-Littlewood maximal inequality for smooth metric measure spaces. 

\begin{prop}[Weak 1-1 Inequality] \label{Weak 1-1 ineq}
Suppose $(M^n, g, e^{-f} \dvol_g)$ with $|f| < k$ has $\Ric_f \geq -(n-1)H$ and $h: M \to \R$ is a nonnegative function.  Define $\Mx_\rho h(p) = \sup_{r \leq \rho} \dashint_{B(p, r)} h e^{-f}\dvol_g$ for $\rho \in (0, 1]$.  Then
 if $h \in L^1(M)$, we have $$\vol_f\{x | \Mx_\rho h(x) > c\} \leq \frac{C(n+4k, H)}{c} \int_M h e^{-f} \dvol_g$$ for any $c > 0$.  
\end{prop}

As in the proof of the Hardy-Littlewood maximal inequality for Euclidean spaces, one utilizes the Vitali Covering Lemma which states that for an arbitrary collection of balls $\{B(x_j, r_j): j \in J\}$ in a metric space, there exists a subcollection of balls $\{B(x_j, r_j): j \in J'\}$ with $J' \subseteq J$ from the original collection which are disjoint and satisfy
$$ \bigcup_{j \in J}B(x_j, r_j) \subseteq \bigcup_{j \in J'} B(x_j, 5r_j).$$
We also note that the $f$-Volume Comparison \cite[Theorem 1.2]{WW} gives a type of doubling estimate.  In particular, for all $r \leq 1$, we have
$$\vol_f(B(x, 5r)) \leq {C(n+4k, H)}{\vol_f(B(x, r))}.$$

\begin{proof} Let $J = \{x | \Mx_\rho h(x) > c\} $.  For all $x \in J$ there exists a ball $B(x, r_x)$ centered at $x$ with radius $r_x \leq 1$  such that 
\begin{equation}\label{ball up bound}\
\int_{B(x, r_x)} he^{-f} \dvol_g \geq c \vol_f B(x,r_x).
\end{equation} 
Then by the Vitali Covering Lemma, we have
$$J \subseteq \ \bigcup_{x \in J}B(x,r_x) \subseteq \bigcup_{x \in J'} B(x, 5r_x)$$
where $J' \subseteq J$.  Then
\begin{equation}\label{max}
\vol_f\{x | \Mx_\rho h(x) > c\} \leq \vol_f \left( \bigcup_{x \in J'} B(x, 5r_x)\right) 
%= \sum_{x \in J'} \vol_f B(x, 5r_x)
 \leq C(n+4k, H) \sum_{x \in J'} \vol_f B(x, r_x).
\end{equation}
Combining \eqref{ball up bound} and \eqref{max} yields the desired result.
\end{proof}

Before continuing to the proof of the theorem, we take a moment to recall Gromov's short generator system.  
To construct a Gromov short generator of the fundamental group $\pi_1(p, M)$, we represent each element of $\pi_1(p, M)$ by a shortest geodesic loop $\gamma$ in that homotopy class.  
A minimal $\gamma_1$ is chosen so that it represents a nontrivial homotopy class of $\pi_1(M)$.  If $\langle \gamma_1 \rangle = \pi_1(M)$, then $\{\gamma_1\}$ is a Gromov short generator system of $\pi_1(M)$.   If not, consider $\pi_1(M) \setminus \langle \gamma_1 \rangle$.  Choose $\gamma_2 \in \pi_1(M) \setminus \langle \gamma_1 \rangle$ to be of minimal length.  If $\langle \gamma_1, \gamma_2 \rangle = \pi_1(M)$, then $\{\gamma_1, \gamma_2\}$ is a Gromov short generator system of $\pi_1(M)$.  If not, choose $\gamma_3 \in \pi_1(M) \setminus \langle \gamma_1, \gamma_2 \rangle$ such that $\gamma_3$ is of minimal length.  Continue in this manner until $\pi_1(M)$ is generated.  By this construction, we obtain a sequence of generators $\{\gamma_1, \gamma_2, \dots\}$  such that $|\gamma_i| \leq |\gamma_{i+1}|$ for all $i$.  The short generators have the property $|\gamma_i| \leq |\gamma^{-1}_j \gamma_i|$ for $i > j$.  Although this sequence of generators is not unique, the sequence of lengths of generators $\{|\gamma_1|, |\gamma_2|, \dots\}$ is unique.

We now have the necessary tools which will allow us to modify the argument of Kapovitch and Wilking to obtain a bound for the number of generators of $\pi_1(M)$ in the smooth metric measure space setting.  As in \cite{KW}, we prove a more general statement from which Theorem \ref{UnifBound} is a consequence.  This general statement, as well as its proof, is parallel to the statement and proof of \cite[Theorem 2.5]{KW}.  The argument is included in its entirety below for completeness.  

\begin{thm} \label{2.5 SMMS}
Given $n$, $k$, and $R$, there is a constant $C$ such that the following holds.  Suppose $(M^n, g, e^{-f}\dvol_g)$ is a smooth metric measure space with $|f|\leq k$, $p \in M$ and $\Ric_f \geq -(n-1)$ on $B(p, 2R)$.  Suppose also that $\pi_1(M, p)$ is generated by loops of length $\leq R$.  Then $\pi_1(M, p)$ can be generated by $C$ loops of length $\leq R$.  

\end{thm}

\begin{proof}[Proof of Theorem \ref{2.5 SMMS}]  

In order to prove Theorem \ref{2.5 SMMS} we begin, as in Kapovitch and Wilking's argument, by showing that there is a point $q \in B(p, \frac{R}{4})$ such that any Gromov short generator system of $\pi_1(M,q)$ has at most $C$ elements.

For $q \in B(p, \frac{R}{4})$ consider a Gromov short generator system $\{\gamma_1, \gamma_2, \dots \}$  of $\pi_1(M, q)$.  By assumption, $\pi_1(M, p)$ is generated by loops of length $\leq R$.  In choosing generators for any Gromov short generator system of $\pi_1(M, q)$, loops of the form $\sigma \circ g \circ \sigma^{-1}$, where $\sigma$ is a minimal geodesic from $q$ to $p$ and $g$ is a generator of length $\leq R$ of $\pi_1(M, p)$, are contained in each of the homotopy classes of $\pi_1(M, q)$.  Such a loop has length $\leq \frac{3R}{2}$ and hence the minimal length representative of that class, $\gamma_i$ must have the property that $|\gamma_i| \leq \frac{3R}{2}$.   
Moreover, there are a priori bounds on the number of short generators of length $\geq r$.  To see this, let us only consider the short generators such that $|\gamma_i| \geq r$.  In the universal cover $\widetilde M$ of $M$, if $\tilde q \in \pi_1^{-1}(q)$, we have 
$$  r \leq d(\gamma_i \tilde q, \tilde q) \leq  d (\gamma_j^{-1} \gamma_i \tilde q, \tilde q) = d(\gamma_i \tilde q, \gamma_j \tilde q)$$
 for $i > j$.
 Thus the balls $B(\gamma_i \tilde q, r/2)$ are pairwise disjoint for all $\gamma_i$ such that $|\gamma_i| \geq r$.  Then,
$$\bigcup_{\{\gamma_i: |\gamma_i| \geq r\}} B(\gamma_i \tilde q , \frac{r}{2}) \subset B(\tilde q, 2R + \frac{r}{2})$$
implies that 
$$\#\{\gamma_i: |\gamma_i| \geq r\} \vol_f B(q, \frac{r}{2}) \leq \vol_f B(q, 2R + \frac{r}{2}).$$
And hence by the volume comparison \cite[Theorem 1.2(a)]{WW}, it follows that $\#\{\gamma_i: |\gamma_i| \geq r\} \leq C(n, k, r, R)$.  Since one can control the number of short generators of length between $r$ and $\frac{3R}{2}$ for $r < R$, one needs only show that the number of short generators of $\pi_1(M, q)$ with length $< r$ can also be controlled.  This argument proceeds by contradiction.  We assume the existence of a contradicting pointed sequence of smooth metric measure spaces $(M_i, p_i)$ such that $(M_i, g_i, e^{-f_i}\dvol_g)$ has the property that 
\begin{itemize}
\item $|f_i| \leq k$
\item $\Ric_{f_i} \geq -(n-1)$ on $B(p_i, 3)$
\item for all $q_i \in B(p_i, 1)$ the number of short generators of $\pi_1(M_i, q_i)$ of length $\leq 4$ is larger than $2^i$.  
\end{itemize}
By the Gromov compactness theorem, we may assume that $(B(p_i, 3), p_i)$ converges to a limit space $(X, p_\infty)$.  Set
$$\dim(X) = \max\{k : \text{ there is a regular } x \in B(p_\infty, \frac{1}{4}) \text{ with } C_x X \simeq \R^k\}$$
where $C_xX$ denotes a tangent cone of $X$ at $x$.  

We prove that there is no such contradicting sequence by reverse induction on $\dim(X)$.  For the base case, let $m > n+4k+1$.  By Lemma \ref{Limit Dimension}, $\dim(X) \leq n+4k+1$, so there is nothing to prove here.  Suppose then that there is no contradicting sequence with $\dim(X) = j$ where $j \in \{m+1, \dots, n+4k\}$ but that there exists a contradicting sequence with $\dim(X) = m$.  The induction step is divided into two substeps.

\vspace{3 mm}
\noindent {\bf Step 1}  For any contradicting sequence $(M_i, p_i)$ converging to $(X, p_\infty)$ there is a new contradicting sequence converging to $(\R^{\dim X}, 0)$.
\vspace{3 mm}

Suppose $(M_i, p_i)$ is a contradicting sequence converging to $(X, p_\infty).$  
 By definition of $\dim(X)$, there exists $q_\infty \in B(p_\infty, \frac{1}{4})$ such that $C_{q_\infty}X \simeq \R^m$.    Let $q_i \in B(p_i, \frac{1}{2})$ such that $q_i \to q_\infty$ as $i \to \infty$.  Since this is a contradicting sequence, it follows that the Gromov short generator systems of $\pi_1(M_i, x_i)$ for all $x_i \in B(q_i, \frac{1}{4})$ contain at least $2^i$ generators of length $\leq 4$.  As noted earlier, for each fixed $\epsilon < 4$, the number of short generators of $\pi_1(M_i, x)$ of length $\in [\epsilon, 4]$ is bounded by a constant $C(n,k,\epsilon, 4)$.  Then we can find a rescaling $\lambda_i \to \infty$ such that for every $x_i \in B(q_i, \frac{1}{\lambda_i})$, the number of generators of $\pi_1(M_i, x)$ of length $\leq 4/\lambda_i$ is at least $2^i$.  Moreover, 
 $(\lambda_i M_i , q_i) \to (\R^m, 0)$, where $\lambda_i M_i$ denotes the smooth metric measure space $(M_i, \lambda_i g_i, e^{-f_i} \dvol_{\lambda_i g_i})$.  Thus the sequence $(\lambda_i M_i, q_i)$ is the new contradicting sequence desired.

\vspace{3 mm}
\noindent {\bf Step 2}  If there is a contradicting sequence converging to $(\R^m, 0)$, then we can find a contradicting sequence converging to a space whose dimension is larger than $m$.  
\vspace{3 mm}

Let $(M_i, q_i)$ denote the contradicting sequence converging to $(\R^m, 0)$ as obtained in Step 1 above.  Without loss of generality, assume that for some $r_i \to \infty$ and ${\epsilon_i} \to 0$, $\Ric_f \geq -{\epsilon_i}$ on $B(p_i, r_i)$.  
By Theorem \ref{9.29 SMMS} there exist $f$-harmonic functions $(\mathbf{b}_1^i, \cdots, \mathbf{b}_m^i): B(q_i, 1) \to \R^m$ such that 
$$ \dashint_{B(q_i, 1)} \left(\sum_{j,l = 1}^m |\langle \nabla \mathbf{b}_l^i, \nabla \mathbf{b}_j^i \rangle - \delta_{lj}| + || \Hess (\mathbf{b}_l^i)||^2\right) e^{-f}\dvol_g < \delta_i \to 0.$$

\noindent{ \bf Claim}  There exists $z_i \in B(q_i, \frac{1}{2})$, $c>0$ such that for any $r \leq \frac{1}{4}$,
$$ \dashint_{B(z_i,r)} \left(\sum_{j,l = 1}^m |\langle \nabla \mathbf{b}_l^i, \nabla \mathbf{b}_j^i \rangle - \delta_{lj}| + || \Hess (\mathbf{b}_l^i)||^2\right)e^{-f}\dvol_g \leq c\delta_i \to 0.$$

\vspace{3 mm}

\noindent Let $h(x)$ denote $ \sum_{j,l = 1}^m |\langle \nabla \mathbf{b}_l^i, \nabla \mathbf{b}_j^i \rangle - \delta_{lj}| + || \Hess (\mathbf{b}_l^i)||^2$ evaluated at $x$.  Seeking contradiction, suppose that for all $c > 0$, $r \leq 1/2$, and $z \in B(q_i, \frac{1}{2})$ 
$$ \dashint_{B(z, r)} he^{-f}\dvol_g > c\delta_i,$$
then it follows that $\Mx_{1/2} h(z) = \sup_{r \leq 1/2} \dashint_{B(z,r)} he^{-f} \dvol_g \geq c\delta_i$.  Hence
\begin{equation} \label{contrary}
 \vol_f \{ x | \Mx_{1/2} h(x) \geq c\delta_i\} \geq \vol_f(B(q_i, \frac{1}{2})).
 \end{equation}
By Proposition \ref{Weak 1-1 ineq}, 
 we also have that for all $c \geq 0$, 
\begin{equation} \label{weak}
\vol_f\{x | \Mx_{1/2} h(x) \geq c\delta_i\} \leq \frac{C(n+4k, -1)}{c}.
\end{equation}
Combining  \eqref{contrary} and \eqref{weak}, we have
$$ 1 \leq  \frac{\vol_f \{ x | \Mx_{1/2} h(x) \geq c\delta_i\}}{\vol_f(B(q_i, \frac{1}{2}))} \leq \frac{C(n+4k, -1)}{c \cdot \vol_f(B(q_i, \frac{1}{2}))}.$$
Choosing $c > C(n+4k, -1)/ \vol_f(B(q_i, \frac{1}{2}))$ yields a contradiction and hence the claim is proven.

By Lemmas \ref{2.2 SMMS} and \ref{2.3 SMMS}, there exists a sequence $\delta_i \to 0$ such that for all $z_i \in B(p_i, 2)$ the Gromov short generator system of $\pi_1(M_i, z_i)$ does not contain any elements of length in $[\delta_i, 4]$.  Choose $r_i \leq 1$ maximal with the property that there is $y_i \in B(z_i, r_i)$ such that the short generators of $\pi_1(M_i, y_i)$ contains a generator of length $r_i$.    Then $r_i < \delta_i \to 0$.  

Rescaling by $\frac{1}{r_i}$ gives that $\pi_1(\frac{1}{r_i} M_i, y_i)$ has at least $2^i$ short generators of length $\leq 1$ for all $y_i \in B(z_i, 1)$.  By the choice of rescaling, there is at least one $y_i \in B(z_i, r_i)$ such that the Gromov short generator system at that $y_i$ contains a generator of length 1. Moreover, the above claim together with the Product Lemma \ref{Product lemma SMMS} give $(\frac{1}{r_i} M_i, z_i) \to (\R^k \times Z, z_\infty)$.  Moreover, by Lemmas \ref{2.2 SMMS} and \ref{2.3 SMMS}, $Z$ is nontrivial and thus $\dim(\R^m \times Z) \geq m+1$, a contradiction.

Thus there exists $q\in B(p, \frac{R}{4})$ such that number of generators of $\pi_1(M, q)$ has at most $C$ elements.   Thus the subgroup of $\pi_1(M,p)$ generated by loops of length $< 3R/5$ can be generated by $C$ elements.  Moreover, the number of short generators of $\pi_1(M,p)$ with length in $[3R/5, R]$ is bounded by some a priori constant.  
\end{proof} 

\vspace{2 mm}

\bibliographystyle{amsalpha}
\bibliography{Jaramillo_SMMS_Fundamental_Group}

\providecommand{\bysame}{\leavevmode\hbox to3em{\hrulefill}\thinspace}
\providecommand{\MR}{\relax\ifhmode\unskip\space\fi MR }
\providecommand{\MRhref}[2]{%
  \href{http://www.ams.org/mathscinet-getitem?mr=#1}{#2}
}
\providecommand{\href}[2]{#2}
\begin{thebibliography}{WW09}

\bibitem[AG90]{AG}
Uwe Abresch and Detlef Gromoll, \emph{On Complete Manifolds with Nonnegative
  Ricci Curvature}, Journal of the American Mathematical Society \textbf{3}
  (1990), no.~2, 355--374.

\bibitem[And90]{And}
Michael~T. Anderson, \emph{Short Geodesics and Gravitational Instantons},
  Journal of Differential Geometry \textbf{31} (1990), no.~1, 265--275.

\bibitem[Bri11]{Bri}
Kevin Brighton, \emph{A Liouville-Type Theorem for Smooth Metric Measure
  Spaces}, Journal of Geometric Analysis (2011), 1--9.

\bibitem[CC96]{CC}
Jeff Cheeger and Tobias Colding, \emph{Lower Bounds on Ricci Curvature and the
  Almost Rigidity of Warped Products}, The Annals of Mathematics \textbf{144}
  (1996), no.~1, 189--237.

\bibitem[CC00]{CC2}
Jeff Cheeger and Tobias~H. Colding, \emph{On the Structure of Spaces with Ricci
  Curvature Bounded Below. II}, Journal of Differential Geometry \textbf{54}
  (2000), no.~1, 13--35.

\bibitem[Che07]{Ch}
Jeff Cheeger, \emph{Degeneration of Riemannian Metrics Under Ricci Curvature
  Bounds}, Publications of the Scuola Normale Superiore Series, Edizioni della
  Normale, 2007.

\bibitem[FLZ09]{FLZ}
Fuquan Fang, Xiang-Dong Li, and Zhenlei Zhang, \emph{Two Generalizations of
  Cheeger-Gromoll Splitting Theorem vie Bakry-Emery Ricci Curvature}, Annales
  de I'institut Fourier \textbf{59} (2009), no.~2, 563--573.

\bibitem[FY92]{FY}
Kenji Fukaya and Takao Yamaguchi, \emph{The Fundamental Groups of Almost
  Nonnegatively Curved Manifolds}, The Annals of Mathematics, Second Series
  \textbf{136} (1992), no.~2, 253--333.

\bibitem[Gro01]{Gro}
Mikhail Gromov, \emph{Metric Structures for Riemannian and Non-Riemannian
  Spaces}, Birkh{\"a}user Boston, 2001.

\bibitem[KW11]{KW}
Vitali Kapovitch and Burkhard Wilking, \emph{Structure of Fundamental Groups of
  Manifolds with Ricci Curvature Bounded Below}, arXiv:1105.5955v2, May 2011.

\bibitem[Lic70]{Li1}
Andre Lichnerowicz, \emph{Varietes Riemanniennes a Tensor c Non Negatif}, C. R.
  Acad. Sc. Paris Serie A \textbf{271} (1970), A650--A653.

\bibitem[Qia97]{Qi}
Zhongmin Qian, \emph{Estimates for Weighted Volumes and Applications}, Quart.
  J. Math. Oxford Ser. \textbf{48} (1997), no.~190, 235--242.

\bibitem[Wei90]{Wei}
Guofang Wei, \emph{On the Fundamental Groups of Manifolds with
  Almost-Nonnegative Ricci Curvature}, Proceedings of the American Mathematical
  Society \textbf{110} (1990), no.~1, 197--199.

\bibitem[Wei97]{WeiBettiNum}
\bysame, \emph{Ricci Curvature and Betti Numbers}, Journal of Geometric
  Analysis \textbf{7} (1997), no.~3, 493--509.

\bibitem[WW07]{WW2}
Guofang Wei and William Wylie, \emph{Comparison Geometry for the Smooth Metric
  Measure Spaces}, Proceedings of the 4th International Congress of Chinese
  Mathematicians, Hangzhou, China \textbf{II} (2007), 191--202.

\bibitem[WW09]{WW}
\bysame, \emph{Comparison Geometry for the Bakry-Emery Ricci Tensor}, Journal
  of Differential Geometry \textbf{83} (2009), no.~2, 377--405.

\bibitem[WZ13]{WZ}
Feng Wang and Xiaohua Zhu, \emph{On the Structure of Spaces with Bakry-Emery
  Ricci Curvature Bounded Below}, arXiv:1304.4490v1, April 2013.

\bibitem[Yun97]{Yun}
Gabjin Yun, \emph{Manifolds of Almost Nonnegative Ricci Curvature}, Proceedings
  of the Mathematical Society \textbf{125} (1997), no.~5, 1517--1522.

\end{thebibliography}

\vspace{2 mm}
\noindent
\textsc{Department of Mathematics,
 University of California,  Santa Barbara, CA 93106}
 \\ {\it E-mail address:}  
 \verb;maree@math.ucsb.edu;

\end{document}